\documentclass[a4paper,10pt,leqno,english]{smfart}
\usepackage{aeguill}
\usepackage{enumerate}
\usepackage{amssymb,amsmath,latexsym,amsthm}
\usepackage[T1]{fontenc}
\usepackage{smfthm}      
\usepackage{geometry}   
\usepackage{url}      
\usepackage[frenchb, english]{babel}
\usepackage[utf8]{inputenc}   
\usepackage{mathrsfs}
\usepackage{xcolor}
\usepackage{comment}
\usepackage{dsfont}
\usepackage{yhmath}
\definecolor{violet}{rgb}{0.0,0.2,0.7}
\definecolor{rouge}{cmyk}{0.0,0.6,0.4,0.3}
\definecolor{rouge2}{rgb}{0.8,0.0,0.2}
\usepackage{hyperref}
\usepackage{hyperref}

\hypersetup{
    bookmarks=true,         % show bookmarks bar?
    unicode=false,          % non-Latin characters in Acrobat’s bookmarks
    pdftoolbar=true,        % show Acrobat’s toolbar?
    pdfmenubar=true,        % show Acrobat’s menu?
    pdffitwindow=false,     % window fit to page when opened
    pdfstartview={FitH},    % fits the width of the page to the window
    pdftitle={Toric plurisubharmonic functions and analytic adjoint ideal sheaves},    % title
    pdfauthor={Henri Guenancia},     % author
    colorlinks=true,       % false: boxed links; true: colored links
   linkcolor=violet,          % color of internal links
    citecolor=rouge2,        % color of links to bibliography
    filecolor=magenta,      % color of file links
    urlcolor=cyan}           % color of external links
\author{Henri Guenancia}
\address{Henri Guenancia \\
Institut de Mathématiques de Jussieu \\
Université Pierre et Marie Curie \\
 Paris}

\email{guenancia@math.jussieu.fr}
\urladdr{www.math.jussieu.fr/~guenancia}

\title{Toric plurisubharmonic functions and analytic adjoint ideal sheaves}

\newcommand{\R}{\mathbb{R}}
\newcommand{\C}{\mathbb{C}}

\newcommand{\N}{\mathbb{N}}

\renewcommand{\d}{\partial}

\newcommand{\db}{\bar{\partial}}

\newcommand{\vp}{\varphi}
\renewcommand{\Im}{\mathscr I \! (\varphi)}
\newcommand{\Imp}{\mathscr I_{\!+} (\varphi)}

\newcommand{\adjz}{\mathcal{A}dj_H^0(\vp)}
\newcommand{\adjzd}{\mathcal{A}dj_D^0(\vp)}

\newcommand{\adj}{\mathcal{A}dj_H(\vp)}
\newcommand{\adjd}{\mathcal{A}dj_D(\vp)}
\newcommand{\adja}{\mathrm{Adj}(\mathfrak a^c,D)}
\newcommand{\adjo}{\mathscr A}
\renewcommand{\a}{\mathfrak a}
\newcommand{\wX}{\widetilde{X}}
\renewcommand{\O}{\mathcal{O}}
\newcommand{\Ox}{\mathcal{O}_{X}}

\newcommand{\Oxt}{\mathcal{O}_{\widetilde X}}
\newcommand{\ep}{\epsilon}
\newcommand{\I}{\!\mathscr I\!}

\newcommand{\la}{\langle}
\newcommand{\ra}{\rangle}

\renewcommand{\ge}{\geqslant}
\renewcommand{\le}{\leqslant}

\newcommand{\twoheadlongrightarrow}{\relbar\joinrel\twoheadrightarrow}

\newcommand{\hg}{\hat g}
\renewcommand{\widering}[1]{\overset{\hbox{\smash{\lower1.333ex\hbox{$\displaystyle\ring{}$}}}}{\wideparen{#1}}}

\newtheorem*{cexem}{Counterexample}
\newtheorem*{thma}{Theorem A}
\newtheorem*{thmb}{Theorem B}
\newtheorem*{thmc}{Theorem C}

\begin{document}

\begin{abstract}In the first part of this paper, we study the properties of some particular plurisubharmonic functions, namely the toric ones. The main result of this part is a precise description of their multiplier ideal sheaves, which generalizes the algebraic case studied by Howald. In the second part, almost entirely independent of the first one, we generalize the notion of the adjoint ideal sheaf used in algebraic geometry to the analytic setting. This enables us to give an analogue of Howald's theorem for adjoint ideals attached to monomial ideals. Finally, using the local Ohsawa-Takegoshi-Manivel theorem, we prove the existence of the so-called generalized adjunction exact sequence, which enables us to recover a weak version of the global extension theorem of Manivel, for compact Kähler manifolds. 
\end{abstract}

\begin{comment}
\begin{altabstract}
Cet article est divisé en deux parties essentiellement indépendantes. Dans la première, on étudie les fonctions plurisousharmoniques toriques, auxquelles on étend le théorème de Howald sur les idéaux multiplicateurs associés à des idéaux monomiaux. Dans la seconde partie, on donne une définition analytique de l'idéal adjoint associé à un idéal, qui s'étend donc à la classe plus large des fonctions plurisousharmoniques. Ceci nous permet de démontrer l'analogue du théorème de Howald pour les idéaux adjoints. Enfin, on démontre l'existence de la suite exacte d'adjonction généralisée, qui, grâce à un résultat d'annulation de type Nadel, nous permet de retrouver le théorème d'extension global de Manivel dans le cadre des variétés kähleriennes compactes.

\end{altabstract}
\end{comment}

\maketitle

\tableofcontents
\section*{Introduction}

\indent \indent Multiplier ideal sheaves are a fundamental tool in complex analytic geometry, for example through Nadel's vanishing theorem: attached to a plurisubharmonic (psh) function $\vp$ on a complex manifold $X$ by $\mathcal J(\vp)_x = \{f \in \mathcal O_{X, x}; ||f||_{\vp}=|f|e^{-\vp} \in L^2_{\mathrm{loc}}(\mathrm{Leb})\}$, they measure the singularity of $\vp$.\\
Lazarsfeld introduced their algebraic analogue for an ideal $\mathfrak a$ using a log-resolution of $\mathfrak a$; of course both ideals coincide whenever $\vp$ is attached to $\a$ (this means that $\vp= \frac 1 2 \log(|f_1|^2+\cdots+|f_r|^2)+O(1)$ if $\a = (f_1, \ldots, f_r)$ locally), but the conceptual gap between the analytic and algebraic definitions suggests that both approaches may be useful, and in some sense complementary to each others.\\ 

In the first part of this paper, we give a different analytic approach of Howald's theorem, that extends to plurisubharmonic functions. \\
More precisely, Howald's theorem states that the multiplier ideal $\mathcal J(\mathfrak a)$ attached to a monomial ideal $\mathfrak a = (\mathbf{z}^{\mathbf{\alpha_1}}, \ldots, \mathbf{z}^{\mathbf{\alpha_r}})\subset \C[z_1,\ldots z_r]$ is generated by monomials $\mathbf{z}^{\mathbf{\alpha}}$ satisfying $\mathbf{\alpha} + \mathds{1} \in \widering{P(\mathfrak a)}$, where $P(\mathfrak a)$ is the Newton polyhedron attached to the $\alpha_i$'s. Extending the notion of Newton polyhedron attached to a monomial ideal to any toric psh function $\vp$ \--- by a toric psh function we mean a psh function that is pointwise invariant under the compact unit torus $\mathbb{T}^n$ \--- and using integrability properties for concave functions, we prove the generalization of Howald's theorem concerning the description of $\mathscr I(\vp)$: 

\begin{thma}
Let $\vp$ be a toric psh function on some polydisk of $\C^n$ centered at $0$, and let us set $\mathds 1 = (1,\ldots, 1)$. Then $\mathscr I(\vp)$ is a monomial ideal, and we have:
\[\mathbf{z}^{\alpha}=z_1^{\alpha_1}\cdots z_n^{\alpha_n} \in \Im \quad  \Longleftrightarrow \quad \alpha + \mathds 1 \in \widering{P(\vp)}.\]
\end{thma}

This shows in passing that the (generalized) openness conjecture stated in \cite{DK} holds for toric psh functions. Finally, we give one example of usual psh function for which we use this result to characterize very precisely the multiplier ideal.\\
Let us notice that J. McNeal and Y. Zeytuncu recently gave a new proof of Howald's theorem in \cite{MZ} using basic analytic techniques. We will use a more systematic approach which points out the crucial role of convexity in this matter.\\

In the second part of this paper, we focus on another ideal sheaf, related to the multiplier ideal sheaf, namely the adjoint ideal sheaf attached to a smooth hypersurface, say $H$. This ideal, well-known in complex algebraic geometry, is a subsheaf of the multiplier ideal sheaf which measures how largely the restricted ideal $\mathcal J(\mathfrak a)_{|H}$ contains $\mathcal J(\mathfrak a_{|H})$, as expressed in \cite{Laz}. 

Our goal is to define an analytic analogue $\adj$ attached to any psh function on a smooth complex manifold $X$. In view of the Ohsawa-Takegoshi-Manivel theorem, the natural candidate for $\adj$ would be defined by its stalks 
\[\adjz_x=\left\{f \in \mathcal O_{X, x}\, ; ||f||_{\vp}\in L^2_{\mathrm{loc}}(\mathrm{Poin}_H)\right\}\]
where $\mathrm{Poin}_H$ is the standard Poincaré volume form attached to $H$; namely if $H$ is locally given by $\{h=0\}$, then $\mathrm{Poin}_H=\frac{1}{|h|^2 \log^2 |h|} \mathrm{Leb}$.\\ 

Unfortunately, the ideal $\adjz$ doesn't coincide in general with the algebraic adjoint: indeed, even in the algebraic case, $\adjz$ fails to satisfy the expected openness property (in general, $\mathcal Adj_H^0((1+\ep)\vp)\neq\adjz$ for any $\ep>0$). So we have to perturb a bit this ideal by setting \[\adj= \bigcup_{\ep >0} \mathcal Adj_H^0((1+\ep)\vp).\]
Once we have introduced our ideal, we need to make sure that it coincides with the algebraic adjoint ideal whenever $\vp$ is associated to an ideal $\mathfrak a = (f_1, \ldots, f_r)$, namely $\vp= \frac 1 2 \log(|f_1|^2+ \cdots+|f_r|^2)+O(1)$ where the $f_i$'s are polynomials (or even holomorphic functions). \\
This new point of view allows us to show an analogue of Howald's theorem for (algebraic) adjoint ideals: 

\begin{thmb}
Let $\a = (\mathbf z^{\alpha_1}, \ldots, \mathbf{z}^{\alpha_k}) \subset \C[z_1, \ldots, z_n]$ be a monomial ideal, $H=\{z_1=0\}$ such that $\a \nsubseteq (z_1)$. We denote by $ri(F_1)$ the relative interior of the (infinite) face of $P(\a)$ contained in $\{x_1=0\}$ and we set $\widetilde{\mathds 1} = (0, 1, \ldots, 1)$. Then for every $c>0$, $\mathrm{Adj}(\a^c, H)$ is a monomial ideal and
\[\mathbf z^{\beta} \in \mathrm{Adj}(\a^c, H) \quad  \Longleftrightarrow \quad \beta + \widetilde{\mathds 1} \in
c \, \cdotp \widering{P(\a)} \, \cup \, c \, \cdotp \mathrm{ri}(F_1).\]
\end{thmb}

We then prove that the fundamental adjunction exact sequence appearing during the proof of Theorem 9.5.1 in \cite{Laz} extends to our setting, under the additional hypothesis that $e^{\vp}$ is Hölder continuous: 
\begin{thmc}
Let $X$ be a complex manifold, $H\subset X$ a smooth hypersurface, and $\vp$ a psh function on $X$, $\vp_{|H} \neq -\infty$, such that $e^\vp$ is locally Hölder continuous. Then the natural restriction map induces the following exact sequence: 
\[ 0 \longrightarrow \Imp \otimes \mathcal O_X(-H) \longrightarrow \adj \longrightarrow \mathscr I_{\!+} (\vp _{|H})\longrightarrow 0\]
\end{thmc}

We deduce from this result that the adjoint ideal is coherent whenever $e^{\vp}$ is locally Hölder continuous, and we then show how to recover a qualitative version of the Manivel global extension theorem on compact Kähler manifolds.

Finally, we give some properties of the sheaf $\adjz$, and explain what can be expected of it.

\section*{Acknowledgments }

I would like to express my sincere gratitude to Sébastien Boucksom whose questions allowed this work to dawn, and whose ideas have contributed in an essential way to its making; may he also be thanked for having made so many times my problems become his own. My gratefulness is also aimed at Paul Guenancia who has always granted me his indefectible trust and his encouragement.

\section{Toric plurisubharmonic functions}

\subsection{Multiplier ideal sheaves}
\label{sec:mis}
In this section, we recall the notion of multiplier ideal sheaves, introduced by Nadel, and which measures the singularity of a psh function. Their definition is rather simple: 

\begin{defi}
Let $X$ be a complex manifold, $\vp$ a psh function on $X$. The multiplier ideal sheaf attached to $\vp$, $\mathscr I (\vp)$, consists in the germs of holomorphic functions  $f \in \mathcal O_{X, x}$ such that $|f|^2e^{-2\vp}$ is integrable with respect to the Lebesgue measure in any local coordinates chart near $x$.
\end{defi}

Let's recall the following fundamental result, even if we won't use it directly (for a proof, see e.g \cite{Dem2}):

\begin{theo}[Nadel, 1989] 
For every psh function $\vp$ on $X$, the sheaf $\mathscr I (\vp)$ is a coherent ideal sheaf on $X$.
\end{theo}

Now we would like to get briefly onto the openness conjecture. We need to recall the definition of a right-regularized version of the multiplier ideal sheaf $\Im$, and introduced in \cite{DEL}: 
\begin{defi}
\phantomsection
\label{def:imp}
Let $X$ be a complex manifold, and $\vp$ a psh function on $X$. We define
\[\Imp = \bigcup_{\ep>0} \mathscr I((1+\ep)\vp)\]
\end{defi}

\begin{rema}
By the strong noetherian property for coherent sheaves, for all $\Omega \Subset X$, there exists $\ep_{\vp, \Omega}>0$ such that for all $0<\ep\leqslant \ep_{\vp, \Omega}$, we have $\Imp_{|\Omega} = \mathscr I((1+\ep)\vp)_{|\Omega} = \mathscr I((1+\ep_{\vp, \Omega})\vp)_{|\Omega}$.
\end{rema}

The famous openness conjecture expressed in \cite{DK} admits a natural generalization in terms of these right-regularized multiplier ideals: 

\begin{conj}[Strong openness conjecture]
\phantomsection
\label{conj:ouvg}
Let $\vp$ be a plurisubharmonic function on $X$, the the following equality of sheaves holds: 
\[ \Imp = \Im.\]
\end{conj}

The only non-trivial case where this conjecture is known is the $2$-dimensional one, as C. Favre and M. Jonsson proved it in their paper \cite{FJ}, using the so-called valuation tree.\\
We now seize the opportunity to discuss briefly the valuative point of view concerning multiplier ideals of psh functions. This approach has been widely developed in \cite{FJ} in the two-variable case, and in \cite{bfj} in higher dimensions. We will only evoke one important result.

We consider a psh germ with isolated singularities at $0\in \C^n$, and we want to describe $\Im$ or $\Imp$ in terms of valuations.
% $v\in \mathcal V_{\mathrm{div}}$: the last object is the set of all divisorial valuations $v$ at $0$, which means that there exists $E\subset \wX$ a prime divisor, with $\mu: \wX \to (\C^n,0)$ birational, such that $v= \orde$. In that case, we define the thinness of $v=\orde$ by $A(v)=1+\orde(J_{\mu})$, where $J_{\mu}$ is the jacobian of $\mu$. In \cite{bfj}, the authors prove the following:
Let us denote by $\mathcal V_{\mathrm{m}}$ the space of monomial valuations, or equivalently Kiselman numbers $v_w$ for $w\in \R_+^n$, defined by: 
\[v_w(\vp)= \sup\left\{\gamma \ge 0, \, \vp(z) \le \gamma  \max_i \left(\frac 1{w_i}\log  |z_i|\right)+\underset{z\to 0}{O}(1)\right\}.\]
For example, $v_w(\mathbf{z}^{\alpha}):=v_w( \log |\mathbf z^{\alpha}|) = \la w, \alpha \ra = \sum w_i \alpha_i$. Note that the thinness of those valuations is: $A(v_w)=|w|= \sum w_i$. 
The following characterization of the multiplier ideal is given in \cite{bfj}: 
\[f \in \Im \quad \Longrightarrow \quad \forall v\in \mathcal V_{\mathrm{m}}\,  , \, \frac{v(\vp)}{v(f)+A(v)}<1.\]
If we consider now the quasi-monomial valuations $v \in \mathcal{V}_{\mathrm{qm}}$ (this means monomial valuations on some birational model of $(\C^n,0)$), one can also define their thinness, and get a full description of $\Imp$:
 \begin{theo}[\cite{bfj}]
\phantomsection
\label{thm:bfj}
Let $\vp$ be a psh germ at $0\in \C^n$. Then 
\[f \in \Imp \quad \Longleftrightarrow \quad \sup_{v\in \mathcal V_{\mathrm{qm}}} \frac{v(\vp)}{v(f)+A(v)}<1.\]
\end{theo}

\subsection{Integrability of the exponential of a concave function}
In the next section, we are going to focus on a very particular type of functions, the toric psh functions. The results we will state about them involve convergence properties for integrals of the form $\int_D e^g$ where $D= \R_+^n$ is the first orthant, and $g$ is any concave function on $D$. So this part is devoted to the study of such integrals.\\
The key-object appears in the following definition: 

\begin{defi}
Let $g$ be a concave function on $D=\{(x_1, \ldots, x_n)\in \R^n, \forall i, x_i \ge 0\}$. The Newton convex body $P(g)$ is:
\[P(g)=\left\{ \lambda \in \R^n; \, g-\la \lambda, \cdotp \ra \le O(1)  \right\}.\]
\end{defi}

\begin{rema}
The set $P(g)$ is the domain of the Legendre transform $g^*(y)=\sup_{x} (g(x)-\la y,x\ra)$.
\end{rema}

It is clear that for any real number $c>0$, $P(cg)=c\cdot P(g)$. Moreover, it is important to notice that $P(g)$ is a convex set, which is in general neither open nor closed (take $g(x)=\frac{-1}{x+1}$ and $g(x)=\log(x+1)$ respectively).\\

Before going into the important results of this section, let us fix some convenient notations:\\
\begin{enumerate}
\item We define a partial ordering on $\R^n$ by \[(x_1,\ldots,x_n)\preceq (y_1,\ldots,y_n) \Longleftrightarrow \forall i \in \{1,\ldots,n\}, \, \,  x_i \le y_i.\] In the same way we define
\[(x_1,\ldots,x_n)\prec (y_1,\ldots,y_n) \Longleftrightarrow \forall i \in \{1,\ldots,n\}, \, \,  x_i < y_i.\]
\item We set $D=\R_+^n$ and $\mathds 1:=(1, \ldots, 1) \in \R^n$.\\
\item We know that the set $E$ of points $v\in D$ such that $g$ is differentiable at $v$ has full measure in $D$ (see \cite{Rob} e.g.). \\
\item We set $Gr(g):=\mathrm{Conv}\left(\left\{\nabla g(v)+\mu; \, (v,\mu) \in E\times (\R_+^*)^n\right\}\right)$.\\
\end{enumerate}

Now take some $\lambda= \nabla g(v)+\mu$ with $(v,\mu) \in E\times (\R_+^*)^n$. Then for every $x\in D$, we have $g(x)-g(v)\le \la \nabla g(v),x-v\ra$ so that $\lambda \in \widering{P(g)}$. By convexity of $\widering{P(g)}$, we thus have $Gr(g)\subset \widering{P(g)}$.
\noindent
The crucial result of this section is given in the next proposition: 

\begin{prop}
\phantomsection
\label{prop:conc}
Let $g$ be a concave function on $D$. Then: 
\[\int_D e^g <+\infty \quad \Longleftrightarrow \quad 0\in \widering{P(g)}.\]
\end{prop}

\begin{proof}[Proof]
The direction  $\Leftarrow$ is easy: there exists $\epsilon >0$ and some constant $C>0$ such that for all $x, \,g(x) +\la \epsilon \mathds 1, x \ra \le C$. Therefore we have:
\begin{eqnarray*}
 \int_{D} e^{g} & \le & C' \int_{D} e^{-\epsilon\sum_i x_i}dx_1\cdots dx_n \\
 &= & C' \prod_{i=1}^n \int_{D} e^{-\epsilon x_i}dx_i \\
 & <& +\infty.
\end{eqnarray*}
As for the other direction, we suppose that the integral $\int_D e^g$ converges. If $0\notin \widering{P(g)}$, by Hahn-Banach's theorem we can find some vector $w\in \R^n$ such that for all $u\in \widering{P(g)}$, we have $ \la u,w \ra \ge 0$ (this implies that $w$ has positive coordinates since $P(g)$ contains a translated of $\R_+^n$). By Fubini's theorem we may find $a\in D$ such that $g$ is differentiable at almost every point of the ray $R=a+\R_+w$ and $\int_R e^g<+\infty.$
\begin{comment}
$A \subset (\mathds 1 +w^{\perp}) \cap \R_+^n$ of full measure (in $\mathds 1 +w^{\perp}$) such that for all $a\in A$, $t \mapsto g(a+tw)$ is integrable on $\R_+$. We claim now that there exists $a_0\in A$ such that for almost all $t\in \R_+$, $g$ is differentiable at $a_0+tw$. Indeed, if $E_a$ denotes the set of points $x \in a+\R_+w$ such that $g$ is not differentiable at $x$, then the measure of the set of points in $A+\R_+w$ at which $g$ is not differentiable is given by 
\[\int_{a\in A} d\lambda_{n-1} \int_{a+\R_+w} \mathds{1}_{E_a} d\lambda_1 = \int_A \lambda_1(E_a)\]
where $\lambda_k$ denotes the Lebesgue measure on $\R^k$. As $g$ is differentiable almost everywhere, this last integral vanishes, so that the sets $E_a$ have measure zero for almost all $a\in A$, which is (more that) what we wanted.
So, we choose such an $a_0$; we know that the oriented ray $R=a_0+\R_+ w$ satisfies $\int_R e^g < \infty$. 
\end{comment}
As $g$ is a $1$-variable concave function on $R$, it is easy to see that $Dg_x(w)$ (Gâteaux-derivative along $w$ at $x$) decreases to some $\ell \in \R \cup \{-\infty\}$ when $x\in R$ tends to infinity. Then the integrability of $e^g$ along $R$ shows that $\ell<0$, so that there exists some $x\in R$, at which $g$ is differentiable, and which satisfies $\la \nabla g(x), w \ra = Dg_x(w)<0$. Thus we can find $\ep>0$ with $\la \nabla g(x)+\ep \mathds 1, w \ra <0$; this is absurd because $\nabla g(x)+\ep \mathds 1 \in Gr(g) \subset \widering{P(g)}$ and the linear form $\la \cdotp ,w\ra$ is non-negative on $P(g)$.
\end{proof}

Rewriting the proof using the open convex set $Gr(g)$ instead of $\widering{P(g)}$, we see that the convergence of $\int_D e^g$ implies that $0\in Gr(g)$. As $P(g+\la \lambda, \cdotp \ra) = \lambda +P(g)$ and $Gr(g+\la \lambda, \cdotp \ra) = \lambda +Gr(g)$ for all $\lambda \in \R^n$, we see that $\widering {P(g)} \subset Gr(g)$. So we have proved:

\begin{prop}
\phantomsection
\label{prop:comp}
For any concave function $g$ on $D$, we have: 
\[Gr(g) = \widering{P(g)}.\]
\end{prop}

%\begin{rema}
%\label{rem:str}
%Proposition \ref{prop:conc} shows that concave functions $g$ on $D$ satisfy the quite striking property: $e^{g} \in L^2(D) \Rightarrow e^{g} \in L^{\infty}(D)$.
%\end{rema}

%As for any convex set $C$ with $\mathring C \neq \varnothing$, we have the equality $\overline{\mathring C} = \overline C$, we obtain in particular: 
%\[ \overline{P(g)}= \overline{\mathrm{Conv}}\left(\left\{\nabla g(v)+\mu; \, (v,\mu) \in E\times \R_+^n\right\}\right).\]

To finish this section, let us stress that what we proved is an openness property; namely if $e^g \in L^1(D)$, then $e^{(1+\ep)g} \in L^1(D)$ for $\ep$ small enough. As any locally uniformly upper bounded sequence of psh functions converging pointwise to a psh function converges in fact in \textit{the} topology of psh functions, the argument given in section 5.4 of \cite{DK} applies here to show that any small perturbation $g+h$, where $h$ is any sufficiently small concave function, satisfies the integrability condition $e^{g+h}\in L^1(D)$.

\subsection{Toric plurisubharmonic functions}

Now we get back to toric plurisubharmonic functions on a polydisk $D(0,r)=\{(z_1,\ldots,z_n) \in \C^n \, | \, \forall i \in \{1,\ldots,n\}  , |z_i|<r\}$. This is a special kind of polydisk (we fix the same radius $r$ for every coordinate), but since all the upcoming results are purely local, there will not be any loss of generality (we could even fix $r=1$). The dimension, $n$, is fixed for the rest of the paper.\\
Let us recall that a toric function $\vp$ on $D(0,r)$ is a function which is invariant under the torus action on $\C^n$: $(e^{i\theta_1}, \ldots, e^{i\theta_n})\, \cdotp z:= (e^{i\theta_1}z_1, \ldots, e^{i\theta_n}z_n)$. In more elementary terms, $\vp(z)$ depends only on $|z_1|, \ldots, |z_n|$. In the psh case, we can say more (e.g. \cite{Dem1}): 
 \begin{prop}
Let $\vp$ be a toric psh function on $D(0,r)$. Then there exists a convex function $f$, non-decreasing in each variable, defined on $]-\infty,\log r[^n$ such that for all $z=(z_1,\ldots,z_n) \in D(0,r)$, we have $\vp(z)=f(\log |z_1|,\ldots,\log |z_n|)$.\\
\end{prop}

For the convenience of the reader, we now state and give an elementary proof of the following well-known result, that will be useful in the following. 
\begin{lemm}
\phantomsection
\label{lem:mon} 
If $\I$ is an ideal of the ring $\mathcal O_{\C^n,0} $ of the germs of holomorphic functions at $0\in \C^n$ such that for every $f\in \I$, all monomials appearing in $f$ are also in $\I$, then $\I$ is generated by monomials (ie it is a monomial ideal).
\end{lemm}

\begin{proof}[Proof]
The first step is to see that, given a (countable) set $I$ of monomials in $n$ variables, we can always extract some finite subset $J$ such that each element of $I$ can be divided by an element of $J$. \\
To see this, we use a \textit{reductio ad absurdum}. So, if this property fails, there exists a sequence $(u_k)_{k\geqslant 1}$ with values in $\mathbb{N}^n$ such that $z^{u_{k+1}}$ cannot be divided by any $z^{u_{p}}$ with $p\leqslant k$. Stated with quantifiers, the property becomes:
\[\exists \, \sigma: \N^2 \to \{1,\ldots, n\} ; \quad  \forall k\geqslant 2, \forall j<k, \, \, (u_{k})_{\sigma(j,k)} < (u_j)_{\sigma(j,k)},\]
where $(u_k)_i$ denotes the $i$-th component of the vector $u_k$.\\
As the sequence $\sigma(k-1,k)$ has values in a finite set, we can extract some subsequence, given by $\psi: \N^* \to \N^*$ increasing, such that $\sigma(\psi(k)-1,\psi(k))$ is a constant, say $1$. 
But then, for every $k\geqslant 2$, we have:  $ (u_{\psi(k)})_{1} < (u_{\psi(k)-1})_{1}$, which is impossible because $(u_k)_1$ is always a non-negative integer. \\
The second step is the result of the lemma itself.\\
As $\mathcal O_{\C^n, 0}$ is noetherian, $\I$ is finitely generated, so we can consider a finite generating family $(f_1, \ldots, f_p)$. For each index $k$, we consider the monomial ideal $\mathcal I_k$ of $\C[z_1, \ldots, z_n]$ generated by the monomials appearing in $f_k$. From the first point, there exists a finite number of minimal monomials appearing in $f_k$, such that the others ones can be divided by the minimal ones. Therefore we have shown that each $f_k$ lies in the ideal of $\mathcal O_{\C^n, 0}$ generated by a finite number of monomials appearing in the expansion of $f_k$. If we put all those minimal monomials for $f_1, \ldots, f_k$ together, we see that $\I$ is generated by (a finite number of) monomials.\\
\end{proof}

For the following, if $\vp$ is a toric psh function on $D(0,r)$ attached to the convex function $f$, we denote by $g$ the concave function defined on $[\log(r),+\infty[^n$ by $g(x)=-f(-x)$. Moreover, if $g$ is attached to $\vp$, we define $P(\vp)$ to be the Newton convex body $P(g)$ of $g$.\\

Now we can state the precise description of the multiplier ideal sheaf attached to any toric psh function, viewed as an ideal $\Im \subset \Ox(X)$ for the Stein manifold $X=D(0,r)$. This can be seen as the analogue or generalization in the analytic setting of Howald's theorem (see \cite{Laz}): 

\begin{theo}
\phantomsection
\label{thm:how2}
Let $\vp$ be a toric psh function on $D(0,r)\subset \C^n$. Then the multiplier ideal $\mathscr{I}(\vp)$ is a monomial ideal, and we have:
\[\mathbf{z}^{\alpha}=z_1^{\alpha_1}\cdots z_n^{\alpha_n} \in \Im \quad  \Longleftrightarrow \quad \alpha + \mathds 1 \in \widering{P(\vp)}.\]
\end{theo}

We want to apply this theorem to the psh function attached to $g=\min_{i} \la \alpha_i, \cdotp \ra$ for some $\alpha_i \in \R^n$. But thanks to Proposition \ref{prop:comp}, $P(\a)$ and $P(g)$ have same interiors, so we obtain: 

\begin{coro}[Howald's theorem, \cite{How}]
Let $\a=(\mathbf{z}^{\alpha_1},\ldots,\mathbf{z}^{\alpha_k})$ be a monomial ideal of $\C[z_1, \ldots, z_n]$, and let $P(\a)$ be the Newton polyhedron attached to the set $\{\alpha_1, \ldots, \alpha_k\}$. Then for every $c>0$:
\[\mathbf z^{\beta} \in \mathcal J(\a^c) \quad \Longleftrightarrow \quad \beta + \mathds 1 \in c \, \widering{P(\a)}.\]

\end{coro}

%\begin{rema}
%J. McNeal and Y. Zeytuncu gave recently another analytic proof of this last result in \cite{MZ}.
%\end{rema}

\begin{proof}[Proof of theorem \ref{thm:how2}]
We first have to check that $\mathscr I (\vp)$ is monomial, so we consider $f=\sum a_{I}z^{I} \in \Im$. This means that for some $r>0$, 
\[\int_{D(0,r)} |f|^2e^{-2\vp(|z_1|,\ldots,|z_n|)}dV(z)\]
is finite. Thanks to Parseval's theorem, this is equivalent to 
 \[ \sum_{I} |a_I|^2 \int_{D(0,r)} |z^I|^2 e^{-2\vp(|z_1|,\ldots,|z_n|)}dV(z) < +\infty\]
 so that each monomial of $f$ already belongs to $\Im$. Then we are done applying lemma \ref{lem:mon}.\\
We are now interested in the convergence of the integral \[\int_{D(0,r)} |z_1|^{2\alpha_1}\cdots|z_n|^{2\alpha_n}e^{-2\vp(z_1,\ldots,z_n)}dV(z).\] 
So we first perform the change of variables variables $z_j=r_je^{i\theta_j}$, and up to a multiplicative factor, the integral equals: \[  \int_{[0,r]^n}r_1^{2\alpha_1+1}\cdots r_n^{2\alpha_n+1}e^{-2f(\log r_1,\ldots,\log r_n)} dr_1\cdots dr_n\]
We set then $t_i=-\log(r_i)$ so that the previous integral becomes \[  \int_{[\log(r),+\infty[^n}e^{-(2\alpha_1+2)u_1}\cdots e^{-(2\alpha_n+2)u_n}e^{2g(u_1,\ldots,u_n)} du_1\cdots du_n\]
or also \[  \int_{[\log(r),+\infty[^n} e^{2g(x)-2 \langle A,x\rangle}dx .\]
Now we just have to apply Proposition \ref{prop:conc} to the concave function $2(g-\la A,\cdot \ra)$, and we are done. 
\end{proof}

As $P((1+\ep)\vp)=(1+\ep) \cdotp P(\vp)$, the characterization of the multiplier ideal given in theorem \ref{thm:how2} implies a (very) particular case of the generalized openness conjecture, recalled in this paper as the conjecture \ref{conj:ouvg}: 

\begin{coro}
The generalized openness conjecture $\mathscr I_{\!+}(\vp)=\mathscr I(\vp)$ holds for any toric psh function $\vp$.
\end{coro}

\begin{rema}
If $X=(\C^n,0)$ and $\mathbf z=z_1\cdots z_n$, then for any toric psh germ $\vp$ and any holomorphic germ $f$ on $X$, Theorem \ref{thm:how2} implies the following property: 
\[f\,e^{-\vp}\in L^2(X) \quad  \Longrightarrow \quad (\mathbf z \,\cdotp  f) \,e^{-\vp} \in L^{\infty}(X).\]
\end{rema}

\subsection{Valuative interpretation}

We now want give the valuative interpretation of Theorem \ref{thm:how2}, keeping in mind the end of section \ref{sec:mis}. \\
For this, let us briefly recall some classical facts about Kiselman numbers, that can be found in \cite{Dem1}. We fix $\vp$ a psh germ at $0\in \C^n$, $w\in \R_+^n$, and we define $\psi_w(z)= \max_i \frac{1}{w_i} \log |z_i|$ and also $\chi_{w}(t)=\sup_{\{\psi_w<t\}} \vp$, which is a convex function. Then we have: 
\begin{eqnarray*}
v_w(\vp) & = & \sup\left\{\mu \ge 0 , \, \vp \le \mu \psi_w +O(1)\right\} \\
 & = & \max\left\{\mu \ge 0 , \, \vp \le \mu \psi_w +O(1)\right\} \\
 & = &\chi_{w}'(-\infty) \\
 & = & \lim_{-\infty} \frac{\chi_{w}(t)}{t}.
\end{eqnarray*}

\begin{defi}
Let $g$ be a concave function on $D=\R_+^n$. Then we define $\hat g$ the homogenization of $g$ on $D\smallsetminus \{0\}$ by
\[\hat g(w) = \lim_{t \to +\infty}  \frac{g(tw)}{t}.\]
\end{defi}
One reason for which we introduced the homogenization function lies in the following lemma: 

\begin{lemm}
Let $\vp$ be a toric psh function on $D(0, r)\subset \C^n$, and $g$ its attached concave function. Then 
\[v_w(\vp)=\hg(w).\]
\end{lemm}

\begin{proof}[Proof]
We write
\begin{eqnarray*}
\chi_{w}(t) &=& \sup \{\vp(z); \, \forall i, \log |z_i|<tw_i \} \\
 & =& \sup \{-g(x); \, \forall i, x_i>-tw_i) \} \quad [x_i:= - \log |z_i|] \\
 & =& - \inf \{g(x); \, \forall i, x_i>-tw_i \}  \\
 & =& -g(-tw)
\end{eqnarray*}
because $g$ is non-decreasing in each variable. Therefore $\frac{\chi_{w}(t)}{t} = \frac{g(-tw)}{-t}$ and passing to the limit when $t \to -\infty$, we obtain the desired result.
\end{proof}

The next result gives a precise description of the closure $\overline{P(g)}$ of the Newton convex body attached to $g$ in terms of $P(\hg)$. 

\begin{lemm}
Let $g_a$ be a non-decreasing in each variable concave function on $D_a=a+\R_+^n$ for some $a \prec 0$. Setting $g={g_a}_{|D}$, we have the following equalities: 
\[\overline{P(g)} = P(\hg) = \{\lambda \in \R^n; \, \hg \le \la \lambda, \cdotp \ra\}.\]
\end{lemm}

\begin{proof}[Proof]
As $\hg$ is homogeneous, if $\lambda \in P(\hg)$, then there exists $C>0$ such that  for all $x\in D\smallsetminus \{0\}$ and all $t>0$, we have $\hg(x)= \frac 1 t \hg(tx) \le \la \lambda, x \ra + \frac C t$ so that when $t$ tends to $+\infty$, we obtain the second identity of the lemma, which shows that $P(\hg)$ is closed.\\
Now, we choose $\lambda \in P(g)$, and write for all $x\in D$, and $t>0$: $\frac 1 t g(tx) \le \la \lambda, x \ra + \frac C t$, and then $\lambda \in P(\hg)$. So we have proved $\overline{P(g)} \subset P(\hg)$.\\
As any convex set with non-empty interior has the same closure than its interior, it is enough to show that $P(g)$ and $P(\hg)$ have the same interior. So we choose $\lambda$ in the interior of $P(\hg)$. This means that there exists $\ep>0$ such that for all $x\in D$, $\hg(x) \le \la \lambda - \ep \mathds 1, x \ra$. We write $x=tw$ where $t>0$ and $w\in \Delta_n = \{(w_1, \ldots, w_n) \in \R_+^n; \, \sum w_i=1\}$ is the standard $n$-simplex, which is obviously compact. An important remark is that $g$ and $\hg$ are restrictions of concave functions $g_a$ and $\widehat{g_a}$ to $D\subset \widering{D_a}$, so they are both continuous on $D$.\\
%We write $g_t(w):=\frac 1 t g(tw)$ for $t>0, w\in \Delta_n$. 
We know that $g$ is non-decreasing in each variable, so $t\mapsto g(tw)$ is a non-decreasing concave function. Clearly, $g$ is bounded below on $D$, thus there exists $C$ such that $h=g+C$ is non-negative on $D$. Then $h_t(w):=\frac 1 t h(tw)$ for $t>0, w\in \Delta_n$, defines a non-increasing (in $t$) family of continuous functions on $\Delta_n$ converging to $\hat{h} =\hg$. Indeed, if $t_2>t_1$, then $ h_{t_1}(w)-h_{t_2}(w) \ge h(0)(\frac 1{t_1}-\frac{1}{t_2})$.\\
By Dini's theorem, the convergence is uniform, as $|h_t(w)-g_t(w)| \le C/t$, then for $t\ge t_0(\ep)$, $||\hg-g_t||_{\Delta_n, \infty} \le \ep$. Thus, for such a $t$, we have:
\begin{eqnarray*} 
g_t(w) &\le& \hg(w)+\ep \\
& \le &   \la \lambda - \ep \mathds 1, w \ra + \ep \\
&= &  \la \lambda , w \ra .
\end{eqnarray*}
If $C = \sup \{g(x)-\la \lambda, x\ra; \, x\in D \, \textrm{and} \, \sum x_i \le t_0\} $, then we have for all $x \in D$, $g(x) \le \la \lambda, x \ra +C$ and therefore $\lambda \in P(g)$. So the interior of $P(\hg)$ is contained in $P(g)$ thus in $\widering{P(g)}$, and as $P(g)\subset P(\hg)$, this concludes the proof of the lemma.
\end{proof}

These two lemmas give now almost immediately the valuative version of Theorem \ref{thm:how2}: 

\begin{theo}
Let $\vp$ be a toric psh germ at $0\in \C^n$. Then $\Im$ is monomial, and: 
\[\mathbf{z}^{\alpha} \in \Im \quad \Longleftrightarrow \quad  \sup_{w \in \R_+^n} \frac{v_w(\vp)}{v_w(\mathbf{z}^{\alpha})+A(w)}<1.\]
\end{theo}

\begin{proof}[Proof]
First of all, we attach to $\vp$ its concave function $g$, and as the singularity is isolated at $0$, we may suppose (by shrinking the domain of $\vp$) that $g$ is the restriction to $D=\R_+^n$ of a concave function on some $D_a=a+\R_+^n$ with $a \prec 0$, so that the preceding lemma applies here, and in particular, $P(g)$ and $P(\hg)$ have same interiors. \\
Thus, using Theorem \ref{thm:how2} and both preceding lemmas, we have: 
\begin{eqnarray*}
\mathbf{z}^{\alpha} \in \Im &\Longleftrightarrow& \alpha + \mathds 1 \in \widering{P(g)}\\
&\Longleftrightarrow& \exists \delta \in ]0,1[; \, \forall w\in D, \, \hg(w) \le (1-\delta)\la \alpha + \mathds 1, w \ra \\
&\Longleftrightarrow& \exists \delta \in ]0,1[; \, \forall w\in D, \, v_w(\vp) \le (1-\delta)\la \alpha + \mathds 1, w \ra \\
&\Longleftrightarrow& \exists \delta \in ]0,1[; \, \sup_{w \in D} \frac{v_w(\vp)}{v_w(\mathbf{z}^{\alpha})+A(w)} \le 1- \delta
\end{eqnarray*}
 which concludes the proof of the theorem.
\end{proof}

\begin{rema}
Compared to Theorem \ref{thm:bfj}, this result tells us that for toric psh functions, multiplier ideals satisfy the openness property, and that they are totally determined by the datum of all \textit{monomial} valuations; namely we don't need to look at divisors lying in some birational model of $(\C^n,0)$ to understand the singularities of toric psh functions. 
\end{rema}

%Indeed, if $\mathbf z^{\alpha} \in \Im$, then we know that for any vector $w\in \R_+^n$, we have $\frac{v_w(\vp)}{v_w(\log |\mathbf z^{\alpha} |)+A(v_w)}<1$, or equivalently $\frac{v_w(\vp)}{\la w, \alpha + \mathds 1 \ra} <1$.

\subsection{An example}
To finish this first part, we illustrate Theorem \ref{thm:how2} with a particular example, for which some computations lead to a rather simple result.\\
Let us define $g(x_1,\ldots,x_n)=k \,x_1^{\alpha_1}\cdots x_n^{\alpha_n}$, with $k>0$ and $\alpha_i \ge 0$ for all $i$. First, we must know whether this function is concave or not. But we can see rather easily that $g$ is concave if and only if $ \alpha_1+ \cdots+\alpha_n \le 1.$\\
Then, following the method suggested by Theorem \ref{thm:how2} and Proposition \ref{prop:comp}, some computations give rise to the following description of the multiplier ideal:
\begin{prop}
Let $\vp(z)=-k\left|\log |z_1|\right|^{\alpha_1}\cdots\left|\log |z_n|\right|^{\alpha_n}$ where the $\alpha_i$ are non-negative real numbers, of sum less or equal than $1$, and $k>0$ be a real number. Then $\vp$ is psh on $D(0,1)$, and:
\begin{enumerate}
\item[$(i)$] Either $\sum \alpha_i <1$, and then $\mathscr I (\vp)=\mathcal O_{D(0,1)};$
\item[$(ii)$] Or $\sum \alpha_i =1$ and then $\mathscr I (\vp)_0$ is generated by the $\mathbf{z}^{\beta}$ such that: 
\[ \prod_{\alpha_i >0} \left(\frac{\beta_i+1}{k\alpha_i}\right)^{\alpha_i}>1.\]
\end{enumerate}

\end{prop}

\section{The analytic adjoint ideal sheaf}

\subsection{Preliminaries}

The adjoint ideal attached to an ideal was introduced in the algebraic setting to deal with extension problems for functions belonging to some multiplier ideals. A general and detailed approach can be found in \cite{Tak} or \cite{Eis}, so we are just going to recall some elementary facts about adjoint ideals. 
\begin{defi}
Let $\a\subset \Ox$ be a non-zero ideal sheaf on a smooth complex variety $X$, $c>0$ a real number, and $D$ a reduced divisor on $X$ such that $\a$ is not contained in any ideal $\mathscr I\!_{D_i}$ of $D_i$ an irreducible component of $D$. We fix $\mu: \wX \to X$ a log resolution of $\a$ such as $\a \cdot \Oxt = \Oxt(-F)$ is such that $F+\mu^*D+K_{\wX/X}+\mathrm{Exc}(\mu)$ is a simple normal crossing divisor. Then the adjoint ideal $\adja$ attached to $c$ and $\a$ is defined by: 
\[\adja = \mu_{*} \Oxt (K_{\wX/X} - [c \cdot F] - \mu^{*} D +D') \]
where $K_{\wX/X}=K_{\wX}-\mu^{*}  K_X, [ \, \, \,]$ denotes the integral part of a divisor, and $D'$ is the strict transform of $D$, defined by linearity. 

\end{defi}

\begin{rema}
To obtain such a resolution, we compose a log resolution $(\mu',X',\mathcal O_{X'}(-F'))$ of $\a$ with a log resolution of $F'+\mu'^*D$. \\
Furthermore, one can show that the previously defined sheaf does not depend on such a log resolution.
\end{rema}

We then have the so-called adjunction exact sequence, appearing (in a hidden way) in the proof of \cite[Theorem 9.5.1]{Laz}: 

\begin{theo}
With the previous notations, and in the case where $D=H$ is a non-singular hypersurface, the following short sequence is exact: 
\[ 0 \longrightarrow \mathscr I (\mathfrak a^c)\otimes \mathcal O_X(-H) \longrightarrow \mathrm{Adj}(\mathfrak a^c,H) \longrightarrow \mathscr I ((\mathfrak a^c) _{|H})\longrightarrow 0\]
\end{theo}

So what we are willing to construct is an analogue of the adjoint ideal which would be attached to any psh function $\vp$. Just as multiplier ideals can be defined using the space of holomorphic germs in $L^2(e^{-\vp}, \mathrm{Leb})$  (it is even their original definition), we would like to find some volume form $\Omega$ such that the space of holomorphic germs in $L^2(e^{-\vp}, \Omega)$ defines adjoint ideals. To find $\Omega$, the intuition is given by the famous Ohsawa-Takegoshi-Manivel theorem (\cite{Dem3}): 

\begin{theo}[Ohsawa-Takegoshi-Manivel]
\phantomsection
Let $X\subset \C^n$ be a bounded pseudoconvex open set, and let $Y\subset X$ be a complex submanifold of codimension $r$, defined by a section $s$ of a hermitian holomorphic bundle with bounded curvature tensor. 
We suppose that $s$ is everywhere transverse to the zero section, and that the inequality $|s|\leqslant e^{-1}$ holds on $X$. 
Then there exists a constant $C>0$ (only depending on $E$) such that: for all psh function $\vp$ on $X$, 
for all holomorphic function $f$ on $Y$ such that $\int_Y |f|^2 |\Lambda^r (ds)|^{-2} e^{-2\vp}dV_{Y} <+\infty$, there exists a holomorphic function $F$ on $X$ extending $f$ such that
\[  \int_{X} \frac{|F|^2}{|s|^2\log^2 |s|}e^{-2\vp}dV_{X} \leqslant C \int_Y \frac{ |f|^2} {|\Lambda^r (ds)|^{2}}e^{-2\vp} dV_{Y} .\]
\end{theo}

So it seems very natural that choosing $\Omega$ to be a Poincaré volume form attached to $H$ (this means that if $H$ is locally given by $\{h=0\}$, then $\mathrm{Poin}_H=\frac{1}{|h|^2 \log^2 |h|} \mathrm{Leb}$) will be the right way to define the analytic adjoint ideal. In this section, we are going to check if things happen as well as predicted.\\\\

Let us now give more general and precise setting. We take a complex manifold $X$ and a simple normal crossing (SNC) divisor $D=\sum D_i$; in the following, we will identify the divisor with its support. Then, for all $x\in X$, there exists a neighborhood $U_x$ of $x$, an integer $0\le p \le n$ and coordinates $z_1, \ldots z_n$ such that $D\cap U_x=\{(z_1,\ldots,z_n)\in U_x; z_1 \cdots z_p=0\} $. In these coordinates, we have obviously:  
\[U_x\smallsetminus D \simeq (\Delta^*)^p \times \Delta^{n-p}, \]
where $\Delta$ is the open unit disk in $\C$, and $\Delta^*$ the punctured disk. If $x \notin D$, then $p=0$.\\

The fundamental object, which is a growth's class of volume forms, is described in the following definition:  
\begin{defi}
Let $X$ be a complex manifold of dimension $n$, $D=\sum D_i$ a simple normal crossing divisor on $X$, and $X_0=X \smallsetminus D$. We say that a positive $(1,1)$-form $\omega_P$ on $X_0$ is $D$-Poincaré if for all sufficiently small open set $U\subset X$ there exists some coordinates $z_1, \ldots , z_n$, $U\cap D =\{(z_1,\ldots,z_n)\in U; z_1 \cdots z_p=0\} $, and some positive constant $C$ such that: 
\[C^{-1}\omega_P \le \frac{i}{2} \left(\sum_{i=1}^{p} \frac{dz_i\wedge d \bar z_i}{|z_i|^2 \log^2 |z_i|}+\sum_{i=p+1}^{n} dz_i \wedge d\bar z_i\right)\le C \omega_P.\]
The associated volume form $\frac{\omega_P^n}{n!}$, which we will denote by $\Omega_P$, is then said to be $D$-Poincaré. So locally, we have up to equivalence:
\[\Omega_P = \prod_{i=1}^p \frac{1}{|z_i|^2 \log^2 |z_i|} \,\, \mathrm{Leb}\, , \]
and the density of $\Omega_P$ is integrable with respect to the Lebesgue measure on $\R^{2n}$.
\end{defi}

\begin{rema}
According to the definition, it is clear that there is a unique $D$-Poincaré volume form on $X$, up to equivalence.
\end{rema}

Let us remark that for a sufficiently small coordinate chart $U\subset X$, if we set $U_0=U\cap X_0$, then the manifold  $\left(U_0, \omega_P= \frac{i}{2} \left(\sum_{i=1}^{p} \frac{dz_i\wedge d \bar z_i}{|z_i|^2 \log^2 |z_i|}+\sum_{i=p+1}^{n} dz_i \wedge d\bar z_i\right)\right)$ is Kähler and weakly pseudoconvex. Indeed, we may suppose $U=D(0,1)$, and then $\vp(z)=-\log(1-|z|^2)-\log |z_1|^2- \, \cdots \, - \log |z_p|^2$ is a smooth psh exhaustion function of $U_0$.

\begin{defi}
\phantomsection
\label{def:iaz}
Let $\vp$ be a psh function on a complex manifold $X$, $D$ an SNC divisor. %such that for all $i$, $\vp_{|D_i} \not \equiv -\infty$. 
We define the ideal sheaf $\adjzd$ to be made up of the germs $f\in \mathcal O_{X,x}$ such that $|f|^2e^{-2\vp}$ is integrable with respect to some (hence any) $D$-Poincaré volume form near $x$.
\end{defi}

\begin{rema}
We always have $\adjzd \subset \mathscr I (\vp)$, and if $x\notin D$, then $\adjzd_x= \mathscr I (\vp)_x$.
\end{rema}

Unfortunately, our sheaf $\adjzd$ fails to coincide in general with the algebraic adjoint, as the following example shows: 

\begin{cexem}
Let $X= (\C^2,0), \, \a = \mathfrak m^6, \, H=\{z_1=0\}$, and $f(z_1,z_2)=z_1^3z_2^3$. If $\vp_{\a}= 3 \log(|z_1|^2+|z_2|^2)$ is a psh function attached to $\a$, then we have:
\[f \in  \mathcal Adj_H(\vp_{\a}) \smallsetminus \mathrm{Adj}(\mathfrak a,H).\]
Indeed, we are in case $(ii)$ of the next Proposition \ref{prop:ex}, with equality in the first large inequality. Therefore, setting $D=D(0,1), \, \int_D \frac{|f|^2e^{-2\vp}}{|z_1|^2\log^2 |z_1|}dV<+\infty$ but for all $\ep>0$, 
\[\int_D \frac{|f|^2}{|z_1|^2 \log^2 |z_1|}e^{-2(1+\ep)\vp}dV=+\infty.\] 
As the algebraic adjoint satisfies the openness property, $f$ cannot belong to $\mathrm{Adj}(\mathfrak a,H)$. 
\end{cexem}

Let us now give the following result that we used in our counterexample, and which gives a precise description of the "zero" adjoint ideal attached to some coordinates monomials:

\begin{prop}
\phantomsection
\label{prop:ex}
Let $\vp = \frac{k}{2} \log(\sum_{i=1}^n |z_i|^{2\alpha_i})$, with $\alpha_i$ some positive real numbers, just as $k$, and let $H$ be the hyperplane $\{z_1=0\}$. Then the stalk at $0$ of $\adjz$ is a monomial ideal, generated by the $z^{\beta}$ satisfying one of the following conditions: 
\begin{enumerate}
\item[$\mathrm(i)$] $\sum \frac{\beta_i+1}{\alpha_i}>k+\frac 1 {\alpha_1} $
\item[$\mathrm(ii)$] $\sum \frac{\beta_i+1}{\alpha_i}\geqslant k+\frac 1 {\alpha_1} \quad \textrm{and} \quad \beta_1 > 0$.
\end{enumerate}
\end{prop}

\begin{proof}[Proof]
The fact that the ideal is monomial can be easily deduced from the same reasoning as the one made to show that multiplier ideals attached to toric psh functions are monomial.\\
We set $N=\sum \frac{\beta_i+1}{\alpha_i}$.\\
As for the computation of the ideal, after a first polar, then toric change of variables, it boils down to the convergence, for $U\subset D(0,\delta)$, $\delta<1$ (resp. $V$) neighborhood of $0$ in $\C^n$ (resp. $\R_+^n$) of the integral: 
\begin{eqnarray*} 
\int_U \frac{\prod_{i=1}^n |z_i|^{2\beta_i}}{|z_1|^2 \log^2 |z_1|  \left(\sum_{i=1}^n |z_i|^{2 \alpha_i}\right)^k} dV_{\C^n} &=& C \int_V \frac{\prod_{i=1}^n r_i^{2\beta_i+1}}{r_1^2 \log^2 r_1 \left(\sum_{i=1}^n r_i^{2 \alpha_i}\right)^k} dV_{\R^n}  \\
&=& C' \int_{t=0}^{\delta} \int_{u\in \mathbb S_{+}^{n-1}} \frac{t^{2(N-k-1/\alpha_1)-1}\prod_{i=1}^n u_i^{2(\beta_i+1)/\alpha_i-1}}{u_1^{2/\alpha_1} \log^2(tu_1)} du \, dt
\end{eqnarray*}
où $\mathbb S_{+}^{n-1} = \{(x_1,\ldots,x_n)\in \R_+^n; x_1^2+\cdots+x_n^2=1\}$.\\
To simplify the computations, we introduce the following notations: $r=2(N-k-1/\alpha_1)-1$, $\lambda_1=2\beta_1/\alpha_1-1$, and for $i\geqslant 2,\lambda_i=2(\beta_i+1)/\alpha_i-1$. 
So we always have $\lambda_1 \geqslant -1$, and for $i\geqslant 2,\lambda_i >-1$. We now have to estimate the following integral: 
\[  I(r,\underline{\lambda})=\int_{t=0}^{\delta} \int_{u\in \mathbb S_{+}^{n-1}} \frac{t^{r}\prod_{i=1}^n u_i^{\lambda_i}}{\log^2(tu_1)} du \, dt\]
An obvious necessary condition of convergence is $r\geqslant -1$, which is equivalent to $N\geqslant k+\frac 1 {\alpha_1}$. \\

$\bullet$ Let us suppose that we have $r>-1$. Then the integral is bounded above by:
\[  \int_{t=0}^{\delta} \int_{u\in [0,1]^{n}} \frac{t^{r}\prod_{i=2}^n u_i^{\lambda_i}}{u_1\log^2(tu_1)} du \, dt \]
and integrating with respect to $u_1$, the last integral becomes:
\[  \int_{t=0}^{\delta} \int_{u\in [0,1]^{n-1}} \frac{t^{r}\prod_{i=2}^n u_i^{\lambda_i}}{-\log  t} du \, dt <+\infty\]
%Ainsi une condition nécessaire est que $2(N-k-1/\alpha_1)-1\geqslant-1$, ce qui est équivalent à $N\geqslant k+\frac 1 {\alpha_1}$.\\
%Comme pour tout $i\geqslant 1$, on a $2(\beta_i+1)/\alpha_i-1>-1$, tout se ramène à la convergence de l'intégrale: 

$\bullet$ We now suppose that $r\geqslant-1$ and $\lambda_1>0$. Then the integral $I(r,\underline{\lambda})$ is less than: 
\[  \int_{t=0}^{\delta} \int_{u\in [0,1]^{n}} \frac{\prod_{i=1}^n u_i^{\lambda_i}}{t\log^2(tu_1)} du \, dt \]
which in turn equals to
\[  \int_{u\in [0,1]^{n}} \frac{\prod_{i=1}^n u_i^{\lambda_i}}{-\log(\delta u_1)} du \, dt <+\infty\]

$\bullet$ Reciprocally, let us assume that $I(r,\underline{\lambda})$ is finite. Thus $r\geqslant -1$, and it remains to show that if $r=-1$, then we necessarily have $\lambda_1>-1$. We use the following equality: 
\[I(-1,\underline{\lambda}) = \int_{u\in \mathbb S_{+}^{n-1}} \frac{\prod_{i=1}^n u_i^{\lambda_i}}{-\log( \delta u_1)} du\]
Then, fixing $\ep = \sqrt 3/2\sqrt{n-1}$, if $u=(u_1,\ldots,u_n)\in  \mathbb S_{+}^{n-1}$ satisfies $u_1 \in [0,\ep]$ and $u_2,\ldots ,u_{n-1}\in [\ep/2,\ep]$, then $x_n\geqslant 1/2$. In fact, $x_n^2\geqslant 1-(n-1)\ep^2=1/4$.\\
So we have the minoration: 
\begin{eqnarray*}
I(-1,\underline{\lambda}) & \geqslant &\int_{u_1=0}^{\ep} \! \int_{u_2=\ep/2}^{\ep}\cdots \int_{u_{n-1}=\ep/2}^{\ep}  2^{-\lambda_n}\frac{\prod_{i=1}^{n-1} u_i^{\lambda_i}}{-\log( \delta u_1)} du_1\cdots du_{n-1} \\
&\geqslant & C \int_{u_1=0}^{\delta \ep} \frac{u_1^{\lambda_1}}{-\log(u_1)} du_{1}
\end{eqnarray*}
where $C$ is a positive constant. But the right hand side is finite if and only if $\lambda_1>-1$, which concludes the proof of the proposition.
\end{proof}

The last counterexample shows us that we have to modify the definition of the analytic adjoint ideal if we want it to extend the usual algebraic adjoint. The goal of the next section is thus to find the correct way to define analytically the adjoint ideal, and to check if this new ideal fits into the generalized adjunction exact sequence.

\subsection{Adjoint ideal attached to a plurisubharmonic function}
\label{sec:comp}
As we saw in the preceding counterexample, our "zero" adjoint ideal doesn't satisfy the expected openness property even in the algebraic case. So the idea is to regularize our ideal: more precisely we suggest the following definition:

\begin{defi}
\phantomsection
\label{def:ia}
With the preceding notations, and those from definition \ref{def:iaz}, we define the analytic adjoint sheaf $\adjd$ to be:
\[\adjd = \bigcup_{\ep>0}\mathcal{A}dj_D^0((1+\ep)\vp) .\]
\end{defi}

In more analytic terms, we can rephrase the definition by saying that $\adjd$ is made up of the germs $f\in \mathcal O_{X,x}$ such that for $\ep>0$ small enough, $|f|^2e^{-2(1+\epsilon)\vp}$ is integrable with respect to any $D$-Poincaré volume form near $x$.\\

A natural question to ask is whether the adjoint ideal is a coherent ideal sheaf. As we will see in Corollary \ref{cor:coh}, if $H$ is a hypersurface of $X$, then it is a consequence of the adjunction exact sequence that the coherence of $\adj$ holds whenever $e^{\vp}$ is locally Hölder continuous. In the general case, one could expect that the coherence should be obtained as the one of the multiplier ideal sheaf, that is solving some well-chosen $\db$-equation with $L^2$ estimates. Unfortunately, a major difficulty appears then, namely solving the $\db$-equation for functions (and not $(n,q)$-forms) with $L^2(\Omega_P)$ estimates, which cannot be deduced from the usual case because of the negativity of the Ricci curvature of the Poincaré metric $\omega_P$.\\

\begin{comment}
Since coherence is checked locally, the next proposition is a straightforward consequence of the previous Proposition \ref{prop1}: 

\begin{prop}
For all psh function $\vp$ on $X$ a complex manifold, and for all SNC divisor $D$ on $X$, the sheaf $\adjd$ is a coherent ideal sheaf.
\end{prop}
\end{comment}

We are now going to show that the sheaf $\adjd$ generalizes the usual adjoint ideal sheaf, in the sense that $\adj$ coincides with the algebraic adjoint ideal whenever $\vp$ has analytic singularities, and that it fits into the adjunction exact sequence. 

\begin{prop}
Let $D$ be an SNC divisor on a smooth complex manifold $X$, $\a$ an analytic ideal sheaf on $X$ not containing any of the ideals of the components of $D$, $c>0$ a real number and $\vp_{c\cdot \a}$ be a psh function attached to $\a^c$. Then the following equality of sheaves holds:
\[ \mathcal Adj_D(\vp_{c\cdot \a}) = \mathrm{Adj}(\a^c,D) .\] 
\end{prop}

\begin{proof}[Proof]
We write
\[\vp =\frac c 2 \log(|f_1|^2+\cdots + |f_N|^2)+O(1)\]
 in the neighborhood of the poles, for $f_i$ local generators of $\a$, and we write $D=\sum_{i=1}^p D_i $.\\
There exists a modification $\mu: X'\to X$, with exceptional divisors $E_1, \ldots , E_{m}$ such that $\mu^*\mathfrak a = \mathcal O_{X'}(-F)$ where $F=\sum_{j=p+1}^{m} a_j E_j$ is such that $F+\mu^*D+K_{X'/X}+\mathrm{Exc}(\mu)$ has simple normal crossings, and satisfies for all $j\geqslant p+1, a_j>0$ (for $j\in \{1,\ldots,p\}$, we set $a_j=0$). Moreover, for $i\in \{1,\ldots,p\}, E_i$ denotes the strict transform of $D_i$. 
To sum up, we use the following notations: 
\begin{eqnarray*}
\mu^*\mathfrak a &=& \sum_{j=p+1}^{m} a_j E_j \\
\mu^* D_i &=& E_i + \sum_{j=p+1}^m b_{i,j} E_j \\
K_{X'}&=&\mu^*K_X + \sum_{j=1}^m c_j E_j
\end{eqnarray*}

We choose $x\in X$, which will be $0$ in our chart. To simplify the notations, we suppose that $p$ is chosen such that $x \in D_1\cap \cdots \cap D_p$. We take the local generators $x_1, \ldots, x_p$ of $\mathcal O_{X}(-D_1), \ldots, \mathcal O_{X}(-D_p)$ respectively. Similarly, $z_k$ will be a local generator of $\mathcal O_{X'}(-E_k)$. \\
If $f$ is a germ of holomorphic function near $x$, defined on a sufficiently small neighborhood $U$ of $0$, we have to compute the following expression: 
\[ \int_{U}\frac{|f|^2 e^{-2(1+\ep)\vp}}{\prod_{k=1}^p |x_k|^2 \log ^2 |x_k|} dV=  \int_{U'=\mu^{-1}(U)}\frac{|f\circ \mu|^2 e^{-2(1+\ep)\vp\circ \mu}}{\prod_{k=1}^p |x_k\circ \mu|^2 \log ^2 |x_k\circ \mu|} \,|J_{\mu}|^2 \, dV' \]
\vspace*{3mm}

Thanks to Parseval's theorem, if a function $f$ is such that the right hand side is finite, then all monomials in the Taylor expansion of $f$ satisfy the same property. So there is no loss of generality in supposing that $f\circ \mu = \prod z_j^{d_j}$. Thus, up to a non-zero multiplicative constant, the right hand side is (we may suppose that $U'$ is contained in a polydisk $D(0,R)$ with $R<1$): 

\[\int_{U'} \frac{\prod_{k=1}^m |z_k|^{2(c_k+d_k-(1+\ep)ca_k)}}{\prod_{k=1}^p \left[|z_k|^2 \log^{2} (|z_k| \prod_{j>p} |z_j|^{b_{k,j}})\right] \cdotp \prod_{k>p} |z_k|^{2e_k} } dV' \]
where we set, for $k>p$, $e_k=\sum_{i=1}^p b_{i,k}$. Setting then $k\in \{1,\ldots,p\}, e_k=1$, the previous integral can be written: 
\[\int_{U'} \frac{\prod_{k=1}^m |z_k|^{2(c_k+d_k-e_k-(1+\ep)ca_k)}}{\prod_{k=1}^p  \log^{2} (|z_k| \prod_{j=1}^m |z_j|^{b_{k,j}})} dV' \]
We set $\lambda_k(\ep)=2(c_k+d_k-e_k-(1+\ep)ca_k)+1$ for all $1\le k \le m$, and changing to polar coordinates leads us to estimate the following integral, on $V$ a neighborhood of $0$ in $\R_+^m$: 
\[I(\ep)=\int_V \frac{\prod_{k=1}^m x_k ^{\lambda_k(\ep)}}{\prod_{k=1}^p \log^2( x_k \prod_{b_{k,j}>0} x_j)}dx_1\ldots dx_m \, \]
and $V\subset B(0,r)$ for some $r<1$. The question of the convergence is answered by the lemma \ref{lem:int} given at the end of the proof.\\

Furthermore, we already know that for $k\in \{1,\ldots,p\}$, we have $\lambda_k(\ep)=2(c_k+d_k-1)+1=2(c_k+d_k)-1\ge -1$. About the condition concerning $k>p$, it is equivalent to: 
\[c_k+d_k\geqslant e_k+[(1+\ep)ca_k].\]
But for all real number $x\geqslant 0$, we have $[(1+\ep)x]=[x]$ for $\ep>0$ small enough, and more precisely for $\ep<([x]+1)/x-1$.\\
Putting all these results together, we have shown that $f\in \adjd$ if and only if for all $k$, we have $d_k\geqslant - (c_k-[ca_k]-e_k)$. Now, let us remind that $\mu^*D-D'= \sum_{k>p} e_k E_k$, so that the previous condition is equivalent to: 
$f\in \mu_{*} \mathcal O_{X'} (K_{X'/X} - [c \cdot F] - \mu^{*} D +D')$, which shows the proposition. \\
\end{proof}

\begin{lemm}
\phantomsection
\label{lem:int}
The integral $I(\ep')$ converges for all $0<\ep'\le\ep$ if and only if for all $k\in \{1,\ldots, m\}$, we have $\lambda_k(\ep) \ge -1$.
\end{lemm}

\begin{proof}[Proof]
The condition is obviously necessary by the usual criterion which determines the integrability near $0$ of $x^{\alpha} \log^{\beta} x$ .\\
Reciprocally, we suppose that for all $k$, we have $\lambda_k(\ep)\ge -1$. Then, as for all $k>p$, we have $a_k>0$, the following inequality holds for all $0<\ep'<\ep$: $\lambda_k(\ep')>-1$. To conclude, we are going to use the identity
\[ \int_{]0,\delta[^{2}} \frac{x^ay^{-1}}{\log^2(xy)} dy dx= \int_0^{\delta} \frac{x^a}{-\log (\delta x)} dx = -\delta^{1-a} \int_0^{\delta^2} \frac{x^a}{\log x} dx\]
in the following computation:
\begin{eqnarray*}
I(\ep')&=&\int_V \frac{\prod_{k=1}^m x_k ^{\lambda_k(\ep')}}{\prod_{k=1}^p \log^2( x_k \prod_{b_{k,j}>0} x_j)}dx_1\ldots dx_m \\
& \le & \int_V \frac{\prod_{k=1}^p x_k^{-1} \prod_{k>p} x_k^{\lambda_k(\ep')}}{\prod_{k=1}^p \log^2( x_k\prod_{b_{k,j}>0} x_j)}dx_1\ldots dx_m \\
& \le & C \int_{V'} \frac{ \prod_{k>p} x_k^{\lambda_k(\ep')}}{|\prod_{k=1}^p \log( \prod_{b_{k,j}>0} x_j)|}dx_{p+1}\ldots dx_m \\
& <  & +\infty
\end{eqnarray*}
where $V'$ is a neighborhood of $0$ in $\R_+^{m-p}$.
\end{proof}

\subsection{Adjoint ideal of a monomial ideal}

We would like to give a precise description of the adjoint ideal attached to a monomial ideal, just as Howald's theorem does for multiplier ideals. Unfortunately, the statement corresponding to the adjoint ideal is a little more complicated.\\
So we work locally and we are given an ideal $\a = (\mathbf z^{\alpha_1}, \ldots, \mathbf{z}^{\alpha_k}) \subset \C[z_1, \ldots, z_n]$, together with the hypersurface $H \subset \C^n$ defined by $\{z_1=0\}$. We know that the Newton polyhedron $P(\a)$ attached to $\a$ has exactly $n$ (infinite) faces $F_1, \ldots, F_n$ which are orthogonal to $e_1=(1, 0, \ldots,0), \ldots, e_n=(0, \ldots, 0, 1) $ respectively, and all other faces of $P(\a)$ are not included in any affine hyperplane $\{x_p= \mathrm{const}\}$. We recall that the relative interior $\mathrm{ri}(F_p)$ of a face $F_p$ is the interior of $F_p$ as embedded in some affine hyperplane $\{x_p=\mathrm{const}\}$. Finally, we define $\widetilde{\mathds 1}:=(0,1, \ldots, 1)$.\\
Now we can state the desired result: 

\begin{theo}
\phantomsection
\label{thm:adjm}
Let $\a = (\mathbf z^{\alpha_1}, \ldots, \mathbf{z}^{\alpha_k}) \subset \C[z_1, \ldots, z_n]$ be a monomial ideal, $H=\{z_1=0\}$ such that $(z_1) \nsubseteq \a$ . Then, for every $c>0$,  $\mathrm{Adj}(\a^c, H)$ is a monomial ideal, and  
\[\mathbf z^{\beta} \in \mathrm{Adj}(\a^c, H) \quad \Longleftrightarrow \quad \beta + \widetilde{\mathds 1}  \in  c \, \cdotp \widering{P(\a)} \, \cup \, c\, \cdotp \mathrm{ri}(F_1).\]
\end{theo}

\begin{proof}[Proof]
We are going to use the analytic definition of $\mathrm{Adj}(\a, H)=\mathcal{A}dj_H(\vp)$ where $\vp$ is attached to the concave function $g=c \, \min_i \la \alpha_i, \cdotp \ra$. Then $\mathbf z^{\beta} \in \adj$ if and only if there exists $\ep>0$ such that on a neighborhood $U$ of $0$ in $\C^n$, the integral
\[\int_V \frac{|\mathbf z|^{2 \beta}e^{-2(1+\ep)g(-\log |z_1|, \ldots, - \log |z_n|)}}{|z_1|^2 \log^2 |z_1|}dV\]
converges. But after performing the usual changes of variables, the convergence of this integral is equivalent to the one of 
\[\int_{[1, +\infty[^n} \frac{e^{2((1+\ep)g(t)-\la A,t \ra )}}{t_1^2}dt_1\cdots dt_n,\]
where $A=\beta + \widetilde{\mathds 1}$. There is no loss of generality in replacing $\alpha_i$ by $c \, \alpha_i$, so we shall suppose that $c=1$ in the following.   \\

$\bullet$ First, we suppose that this integral converges for some $\ep>0$. This implies that for all $\eta>0$, the integral 
\[\int_{[1, +\infty[^n} e^{2((1+\ep)g-\la A+(\eta, 0 , \ldots, 0), \cdotp \ra )}dt\]
converges, so that, thanks to Proposition \ref{prop:conc}, we have:
\begin{equation}
\label{eq:int}
\forall \eta>0, \quad A+(\eta, 0 , \ldots, 0) \in (1+\ep) \widering{P(g)}\\
\end{equation}
 We claim that (\ref{eq:int}) is equivalent to 
\begin{equation}
\label{eq:int2}
A \in \widering{P(\a)} \cup \mathrm{ri}(F_1)
\end{equation}

The implication (\ref{eq:int2}) $\Rightarrow$ (\ref{eq:int}) is clear because as $(z_1) \nsubseteq \a$,  $F_1\subset \{x_1=0\}$ contains thus the infinite face orthogonal to $e_1$ attached to $(1+\ep)P(\a)$, so that (\ref{eq:int}) holds.\\
As for the other direction, we first show that if $A$ belongs to some face $F$ of $P(\a)$ which is not one of the $F_i$'s ($i\ge 1$), then for all $\ep>0$, (\ref{eq:int}) fails to be true. Indeed, as each $\alpha_i$ has non-negative components, $F$ is included in some affine hyperplane $\{x; \la x,w \ra = \alpha \}$ where $w\neq 0$ has non-negative components, and $\alpha > 0$ (because $F$ is not one of the $F_i$'s). Thus we have $(1+\ep)\widering{P(\a)} \subset \{x; \la x,w \ra > (1+\ep) \alpha \ra\}.$\\
Therefore we should have for all $\eta>0$: $\la A+(\eta, 0, \ldots, 0), w \ra > (1+\ep) \alpha$, or equivalently $\eta w_1 > \ep \alpha$, which is absurd because $\eta$ can be arbitrarily small.\\
The case where $A$ belongs to one of the faces $F_2, \ldots, F_n$ is immediate, so we have proved that if $\mathbf z^{\beta} \in \adj$, then (\ref{eq:int2}) holds.\\

$\bullet$ Conversely, if $A \in \widering{P(\a)} $, then $A \in (1+\ep)\widering{P(\a)} $ for $\ep>0$ sufficiently small, so that $e^{2[(1+\ep)g - \la A, \cdotp \ra]}$ is integrable. In the case where $A \in \mathrm{ri}(F_1)$, then there exists some $\lambda=(0,\lambda_2, \ldots, \lambda_n) \in \R \times (\R_+^*)^{n-1}$ and some barycentric coefficients $t_i$ such that $A=\sum t_i \alpha_i + \lambda$. As $g\le \sum t_i \alpha_i$, we have: $(1+\ep)g - \la A, \cdotp \ra \le \la \ep A - \lambda, \cdotp \ra.$ As $F_1\subset \{z_1=0\}$, the first component $A_1$ of $A$ is zero, so that if we choose $0< \ep < \min \{\frac{\lambda_i}{2A_i}; i \ge 2\}$, then 
\[e^{(1+\ep)g(t) - \la A, t \ra} \le e^{-\lambda_2t_2/2} \cdots e^{-\lambda_nt_n/2}\] 
and thus the integral 
\[\int_{[1, +\infty[^n} \frac{e^{2((1+\ep)g(t)-\la A,t \ra )}}{t_1^2}dt_1\cdots dt_n\]
is convergent. Therefore $\mathbf z^{\beta} \in \adj$, which concludes the proof of the theorem.
\end{proof}

\begin{rema}
The polyhedral structure of $P(\a)$ is thus crucial to read the adjoint ideal of $\a$ on its Newton polyhedron. Therefore, as the Newton convex body attached to a toric psh function $\vp$ is in general not polyhedral (take $g$ \---the concave function attached to $\vp$\--- to be any non affine smooth concave function, and use the description of $P(g)$ in terms of gradients given in Proposition \ref{prop:comp}), it seems hopeless to generalize Theorem \ref{thm:adjm} to the general toric psh case.
\end{rema}

\subsection{The adjunction exact sequence}
We turn now to the generalized adjunction exact sequence. To prove the validity of the adjunction exact sequence in the analytic setting, we are going to use in a essential manner the proof of the so-called inversion of adjunction, that we may find in \cite{DK}. The main difficulty we will face is to show that the restriction map is well-defined. As for seeing that this restriction induces a surjection, it will be a straightforward consequence of the Ohsawa-Takegoshi-Manivel theorem.\\
Before going into the proof, we give an easy but useful result:

\begin{lemm}
\phantomsection
\label{lem:se}
Let $\Omega\subset \C^n$ an open set that is relatively compact in the unit polydisk, let $\vp$ a psh function on $\Omega$ such that for all $z\in \Omega, \vp(z)\leqslant -1$, let $f$ be an holomorphic function on $\Omega$, and $\alpha>0$ a real number. \\
If there exists $\ep>0$ such that $\int_{\Omega} \frac{|f|^2e^{-2(1+\ep)\vp}}{(-\vp)^\alpha} dV_{\Omega}$ converges, then there exists $\ep'>0$ such that the integral
$ \int_{\Omega} |f|^2e^{-2(1+\ep')\vp}dV_{\Omega}$ converges.\\
In particular, if $\int_{\Omega} \frac{|f|^2e^{-2(1+\ep)\vp}}{\log^2 |z_n|} dV_{\Omega}$ converges, then $\int_{\Omega'} |f|^2e^{-2(1+\ep')\vp}dV_{\Omega}$ converges too, for some $\ep'>0$ and all $\Omega'\Subset\Omega$.
\end{lemm}

\begin{proof}[Proof]
We set $C =\inf \{e^{\ep x}/x^{\alpha}; x\geqslant 1\}$, it is a positive number. Then the inequality 
\[\int_{\Omega} \frac{|f|^2e^{-2(1+\ep)\vp}}{(-\vp)^\alpha} dV_{\Omega} \geqslant C  \int_{\Omega} |f|^2e^{-2(1+\ep/2)\vp}dV_{\Omega}\]
 shows the first assertion. \\
As for the second, we define $A=\{z\in \Omega; \vp(z) \leqslant \frac 1 4 \log |z_n|\}$ and $B = \{z\in \Omega; \vp(z) \geqslant \frac 1 4 \log |z_n|\}$. Then
\[\int_{A} \frac{|f|^2e^{-2(1+\ep)\vp}}{\log^2 |z_n|} dV_{\Omega} \geqslant \int_A \frac{|f|^2e^{-2(1+\ep)\vp}}{16 \,  \vp^2} dV_{\Omega}\]
and using the first part, this implies that $\int_{A} |f|^2e^{-2(1+\ep')\vp}dV_{\Omega}$ is finite for some $\ep'>0$.\\
Furthermore, setting $\delta = \min (\ep',1)$, the following inequality holds on $B$: $-2(1+\delta) \vp \leqslant -(1+\delta)/2 \log |z_n|$, thus: 
\[\int_{B\cap \Omega'}  |f|^2e^{-2(1+\delta)\vp}dV_{\Omega} \leqslant ||f||_{L^{\infty}(\Omega')} \int_{\Omega} |z_n|^{-\frac{1+\delta}{2}}dV_{\Omega}<+\infty\]
which concludes the proof of the lemma.
\end{proof}

Now we can prove the main result of this section:

\begin{theo}
\phantomsection
\label{thm:se}
Let $X$ be a complex manifold, $H\subset X$ a smooth hypersurface, and $\vp$ a psh function on $X$, $\vp_{|H} \neq -\infty$, such that $e^\vp$ is locally Hölder continuous, and let $i:H\hookrightarrow X$ be the inclusion. The the natural restriction map induces the following exact sequence: 
\[ 0 \longrightarrow \Imp \otimes \mathcal O_X(-H) \longrightarrow \adj \longrightarrow i_*\mathscr I_{\!+} (\vp _{|H})\longrightarrow 0\]
\end{theo}

\begin{proof}[Proof]
What we have to check is that the restriction map is well-defined, that it is surjective, and that this sequence is exact. We proceed in the order we just described.
As everything is purely local, we may assume that $H$ is the hyperplane $z_n=0$ in the polydisk $U=D(0,r),r<1$ in $\C^n$. Moreover, since changing $\vp$ into $\vp-C$ does not affect the questions of integrability, and since $\vp$ is locally upper bounded, we may assume that $\vp\leqslant -1$ on $U$, so that we can apply the preceding lemma \ref{lem:se}.\\
So, we choose a holomorphic non-zero function $F$, defined on a neighborhood $U$ of $0$, and satisfying $F \in \adj(U)$. We write then $F(z)=F(z',z_n)=(F(z',z_n)-F(z',0))+F(z',0)$, and as $F$ is holomorphic, there exists a constant $C_1>0$ such that
$|F(z',0)|^2 \leqslant C_1 |z_n|^2+|F(z)|^2$ and therefore $|F(z)|^2\geqslant |F(z',0)|^2 - C_1 |z_n|^2$.\\
Furthermore, as $e^{\vp}$ is Hölder, there exists $\alpha \in ]0,1]$ and $C_2>0$ such that
\[e^{2\vp(z)}\leqslant\left( e^{\vp(z',0)}+C_2 |z_n|^{\alpha}\right)^2\leqslant C_3 (e^{2\vp(z',0)}+ |z_n|^{2\alpha}) \] 
with $C_3=4\max(1,C_2)$. Setting $f(z')=F(z',0)$, we obtain the following inequalities: 
\begin{eqnarray*}
 \frac{|F(z)|^2e^{-2(1+\ep)\vp(z)}}{|z_n|^2 \log^2 |z_n|} &\geqslant & C_3^{-1} \frac{|F(z)|^2}{\log^2 |z_n|}\cdotp \frac 1 {|z_n|^2(e^{2\vp(z',0)}+ |z_n|^{2\alpha})^{1+\ep}}\\
 & \geqslant & \frac{C_3^{-1}|f(z')|^2}{|z_n|^2 \log^2 |z_n|(e^{2\vp(z',0)}+ |z_n|^{2\alpha})^{1+\ep}}\\
 &&-\frac{C_3^{-1}C_1}{\log^2 |z_n|(e^{2\vp(z',0)}+ |z_n|^{2\alpha})^{1+\ep}}
\end{eqnarray*}
Now we suppose that  $U=U'\times D(0,r_n)$ (if it's not the case, we just have to restrict $U$ a bit), and we partially integrate with respect to the last variable, in the family of disks $|z_n|<\rho(z')$ with $\rho(z')=\delta e^{(1+\ep)\alpha^{-1}\vp(z',0)}$ where $\delta>0$ is small enough so that $\rho(z')<r_n$ for all $z'\in U'$. \\
%Comme on a $\log |z_n| < \log \epsilon + \alpha^{-1} \vp(z',0) \leqslantC_4$ pour la même raison, et $e^{2\vp(z',0)}+C_2 |z_n|^{2\alpha} \leqslantC_5 e^{2\vp(z',0)}$, l'intégration partielle donne: 
The right term in the right hand side is easily estimated when integrated, because $\log^2|z_n|\geqslant \log^2 r >0$, $z_n$ being of module $\leqslant r <1$, we have:
\[\int_{|z_n|<\rho(z')}\frac{C_1}{\log^2 |z_n|(e^{2\vp(z',0)}+ |z_n|^{2\alpha})^{1+\ep}}dV(z_n) \leqslant C_4 \delta^2 e^{(\frac 2 {\alpha} -2)(1+\ep)\vp(z',0)} \]
which is bounded because $\alpha \leqslant1$.\\
%En revanche, on ne peut pas minorer par une constante le dernier terme de l'inégalité, il faut être plus précis. Pour cela, on écrit: 
As for the remaining term, we write: 
\begin{eqnarray*}
\int_{|z_n|<\rho(z')} \frac {dV(z_n)} {|z_n|^2 \log^2 |z_n|(e^{2\vp(z',0)}+ |z_n|^{2\alpha})^{1+\ep}}  & \geqslant & C_5\int_{|z_n|<\rho(z')} \frac{e^{-2(1+\ep)\vp(z',0)}dV(z_n)}{|z_n|^2 \log^2 |z_n|}\\
& \geqslant & C_6 e^{-2(1+\ep)\vp(z',0)} \int_0^{\rho(z')} \frac{dt}{t \log^2 t}\\
& = & -C_6 \frac{e^{-2(1+\ep)\vp(z',0)}}{\log \rho(z')}
\end{eqnarray*}
Then, as $\log \rho(z') = \log \delta + (1+\ep)\alpha^{-1} \vp(z',0)$, the lemma \ref{lem:se} gives the expected result (instead of integrating, we could have written directly $(|z_n|^2 \log^2 |z_n|)^{-1} \ge (\rho(z')^2 \log^2 \rho(z'))^{-1}).$  \\
To show the surjectivity of the last map, we use the local version of Ohsawa-Takegoshi-Manivel, with the weight $(1+\ep)\vp$.\\
Finally, to show that the sequence is exact, if $f\in \adj$ vanishes on $H\cap U$, then we write locally $f=g\cdot z_n$ where $g$ is holomorphic, and satisfies on an open set $W \subset U$: 
\[ \int_{W} \frac{|g|^2e^{-2(1+\ep)\vp}}{\log^2 |z_n|} dV<+\infty\]
and using again lemma \ref{lem:se}, we can conclude that $g\in \Imp(W)$, which had to be proved.
%\[\int_{|z_n|<\rho(z')} \frac {dV(z_n)} {|z_n|^2 \log^2 |z_n|(e^{2\vp(z',0)}+ |z_n|^{2\alpha})} \geqslant C_8 e^{-2\vp(z',0)}.\]
%\[\int_{U'}|f(z')|^2 e^{-2\vp(z',0)}dV(z') \leqslantC_6 \left(\int_U \frac{|F(z)|^2e^{-2\vp(z)}}{|z_n|^2 \log^2 |z_n|} dV(z)+\int_{U'}\right)\]
\end{proof}

\begin{rema}
In the case where $e^\vp$ is not Hölder continuous, the restriction map may not be well-defined anymore: on the polydisk of radius $\frac 1 2$ in $\C^2$, we choose $f=1$ and \[\vp(z_1,z_2)= \max(-\lambda \log(-\log |z_1|),\log |z_2|)\] with $0<\lambda<\frac 1 2$. We then have $\vp(z)\geqslant -\lambda \log (-\log |z_1|)$ thus $\frac{e^{-2\vp(z)}}{|z_1|^2\log^2|z_1|} \leqslant \frac{1}{|z_1|^2\left|\log |z_1|\right|^{2(1-\lambda)}}$ which is integrable on the polydisk, and it is then easy to see that any constant function belongs to $\adj$. However on the hyperplane $\{z_1=0\}$, $e^{-2\vp(z)}=|z_2|^{-2}$ is not integrable, so that any non-zero constant function does not belong to $\mathscr I_+(\vp_{|H})$, or even to $\mathscr I(\vp_{|H})$.
\end{rema}

\begin{rema}
If $\vp$ has analytic singularities in the sense that near the poles, $\vp=\log(|f_1|+\cdots+|f_r|)+v$ where the $f_i$'s are holomorphic and $v$ is smooth, we know that $\Imp=\Im$, and we have proved previously that $\adj$ coincide with the algebraic ideal. Moreover, $e^\vp$ is clearly Hölder continuous near the poles, so Theorem \ref{thm:se} is a generalization of the algebraic adjunction exact sequence given in \cite{Laz}.
\end{rema}

\begin{coro}
\phantomsection
\label{cor:coh}
Let $X$ be a complex manifold, $H\subset X$ a smooth hypersurface, and $\vp$ a psh function on $X$, $\vp_{|H} \neq -\infty$, such that $e^\vp$ is locally Hölder continuous. Then $\adj$ is a coherent ideal sheaf on $X$.
\end{coro}

We now turn to get a global extension theorem for holomorphic forms on a hypersurface with coefficients in some multiplier ideal. For this, we only need some vanishing result for cohomology groups, and therefore it is practical to introduce some notations for the global (compact) setting.

\begin{defi}
Let $X$ be a complex manifold, and $H$ a hypersurface of $X$. We pick an almost-psh function $\vp$ (ie it is locally the sum of a psh function and of a smooth function) non identically equal to $-\infty$ on $H$, $T$ a positive closed current of bidegree $(1,1)$ on $X$ well-defined on $H$, and $h$ a singular hermitian metric of some holomorphic line bundle, satisfying $h_{|H}\not \equiv +\infty$. The we define:
\begin{enumerate}
\item[$\bullet$] If locally, $\vp=\psi+f$ with $\psi$ psh and $f$ smooth, then we set $\adj:=\mathcal A \mathrm{dj}_H(\psi)$, which makes sense globally;
\item[$\bullet$] If locally $T=S+dd^c \vp$ where $S$ is smooth, and thus $\vp$ is almost-psh, we set $\mathcal A \mathrm{dj}_H(T):=\adj$, which makes sense globally;
\item[$\bullet$] If the metric $h$ has an almost positive curvature current $\Theta_h$, we set $\mathcal A \mathrm{dj}_H(h):=\mathcal A \mathrm{dj}_H(\Theta_h)$.
\end{enumerate}
\end{defi}

Of course, we can make the same definition with the multiplier ideals instead of the adjoint ideals.\\\\
Combining the adjunction exact sequence and a variant of Nadel vanishing theorem, we can give a global result for extending holomorphic functions with some finite $L^2$ norms. So we have to prove the following result: 

\begin{prop}
Let $(X,\omega)$ be a compact Kähler manifold, and let $(E,h)$ be a line bundle on $X$, where $h$ is a singular hermitian metric, whose curvature tensor $T$ satisfies $T\ge \eta \,  \omega$ for some $\eta >0$. Then 
\[\forall q>0, \quad H^q(X, \mathcal O_X(K_X+E)\otimes \mathscr I_+(h))=0.\]
\end{prop}

\begin{proof}
We know that there exists $\ep_0>0$ such that $\mathscr I_+(T)=\mathscr I((1+\ep_0)T)$. If $h_0$ is any smooth metric on $E$, of curvature current $T_0$, we can choose $\ep_0\ge\ep>0$ sufficiently small such that $(1+\ep)T-\ep T_0 \ge \eta/2$. Then we just have to apply Nadel's vanishing to the hermitian line bundle $(E,h^{1+\ep}\otimes h_0^{-\ep})$ whose multiplier ideal is precisely $\mathscr I_+(h)$.
\end{proof}

\begin{rema}
In view of the openness conjecture, it seems very natural that this result should hold in the general weakly pseudoconvex (Kähler) case, though we were not able to prove it.
\end{rema}

\begin{coro}
Let $(X,\omega)$ be a compact Kähler manifold, $H\subset X$ a smooth hypersurface, $(E,h)$ a holomorphic line bundle equipped with a singular hermitian metric $h$, $h_{|H}\not \equiv+\infty$, whose curvature current has local potentials $\vp$ such that $e^{\vp}$ is Hölder continuous, and such that there exists some $\eta>0$ satisfying $i\d \db \vp \geqslant \eta \, \omega$.\\
Then every holomorphic section $s\in H^0(H,\mathcal O_H(K_H+E_{H})\otimes \mathscr I \!_+(h_{|H}))$ extends to a section $\tilde s \in H^0(X,\Ox(K_X+H+E)\otimes \mathcal A \mathrm{dj}_H(h))$.
\end{coro}

\begin{rema}
It is well-known that if a compact Kähler manifold carries a integral Kähler current, then it is Moishezon and therefore automatically projective.
\end{rema}

\begin{proof}[Proof]
Tensorizing the adjunction exact sequence by $K_{X}+E+H$, we obtain :
\[ 0 \longrightarrow \I_+(h) \, \otimes \, \Ox(K_{X}+E) \longrightarrow \mathcal A \mathrm{dj}_H(h) \, \otimes \, \Ox(K_{X}+E+H) \longrightarrow i_*\mathscr I_{\!+} (h_{|H}) \, \otimes \, \mathcal O_H(K_H+E_H)\longrightarrow 0\]
If $T$ is the Chern curvature of $(E,h)$, then the last proposition show that:  
\[H^1(X,\Ox(K_X+E) \otimes \I_+(h))=H^1(X,\Ox(K_X+E) \otimes \I_+(T))=0 \]
Therefore the restriction maps induces a surjection
\[ H^0(X, \Ox(K_{X}+E+H)\otimes \mathcal A \mathrm{dj}_H(h))  \twoheadlongrightarrow H^0(H,  \mathcal O_H (K_H+E_H) \otimes \mathscr I_{\!+} (h _{|H}))\] 
which had to be proved.
\end{proof}

The approach we used to show this result, which relies in an essential manner on the local version of Manivel's theorem, is a natural way to obtain the global version of Manivel's theorem \cite{Dem3}. Nevertheless, the result we obtain is a quite weaker version of the original Manivel's theorem, in the sense that it is qualitative (we don't have any control on the $L^2$ norm anymore), and is only given for "regular" currents (more precisely with Hölder psh local potentials).

\subsection{Back to the sheaf $\adjz$}

Finally, we would like to give one positive result concerning the ideal $\adjz$. For this, the crucial fact is given in the following lemma, which assumptions are unfortunately really restrictive:

\begin{lemm}
\phantomsection
\label{lem:cru}
Let $\vp$ be a psh function which has only analytic or toric singularities on a bounded open set $B\subset \C^n$, satisfying $\vp < C < 0$ on $B$, and let $f$ be an holomorphic function on $B$. If the integral
\[\int_B |f|^2 \frac{e^{-2\vp}}{\vp} dV\] 
converges, then so does the integral
\[\int_B|f|^2 e^{-2\vp}dV.\] 
\end{lemm}

\begin{proof}[Proof]
We start with the case where $\vp$ has analytic singularities. As $B$ is bounded, the question is actually local, and using a log resolution, and using the notations of section \ref{sec:comp}, the integrability assumption becomes:
\[\int_{U'} \frac{\prod|z_j|^{2(c_j+d_j-a_j)}}{|\log (\prod |z_j|^{a_j})|} dV'<+\infty, \]
and performing a toric change of variable, the Bertrand criterion shows that
\[\int_{U'} \prod|z_j|^{2(c_j+d_j-a_j)}dV'<+\infty, \]
which we had to show.\\

Let us now get to the toric case. Again, the question is local, so we borrow the techniques of the first part. The calculus appearing in the proof of Theorem \ref{thm:how2} shows that we are boiled down to show the integrability of $e^h$ for some $h$ concave (more precisely, $f$ can be chosen to be a monomial $z^{\alpha}$, and $h=g-\la \alpha, \cdotp \ra$ for $g$ the concave function attached to $\vp$). As for non-zero concave function $h$ of $1$-variable (say on $\R^+$) the integrability of $e^h/g$ for $g$ a non-zero concave function implies the one of $e^h$ (indeed, $h(x)=O_{x\to+\infty}(x)$ or more precisely, either $h(x)$ tends to $0$ when $x$ goes to $+\infty$, or it is equivalent to $\ell x$ where $\ell=\lim \frac{d^+}{dx}h(x,\cdot)$ is non-zero; idem for $g$), we can follow the proof of Proposition \ref{prop:conc} to show that this extends to all dimensions.

\end{proof}

Denoting $\widetilde{\mathscr I(\vp)}$ the analogue of the multiplier ideal sheaf where we replace the integrability condition by the local integrability of $\frac{|f|^2e^{-2\vp}}{\log^2 |s|}$, where $s$ is a (local) section defining $H$, with $ds_{|H}$ never zero. Then we have the following result: 

\begin{theo}
\phantomsection
\label{prop:se}
Let $\vp$ be a Hölder psh function whose singularities are only analytic or toric, and let $i:H\hookrightarrow X$ the inclusion. The the natural restriction map induces the following exact sequence: 
\[ 0 \longrightarrow \widetilde{\mathscr I(\vp)} \otimes \mathcal O_X(-H) \longrightarrow \adjz \longrightarrow i_*\mathscr I (\vp _{|H})\longrightarrow 0\]
%In particular, under these assumptions, $\widetilde{\mathscr I(\vp)}$ is a coherent ideal sheaf.
\end{theo}

\begin{proof}[Proof]
The proof is very similar to the one of the adjunction exact sequence. The only difference appearing here concerns the restriction map, which has \textit{a priori} no reason to be well-defined.\\
So we take $F\in \adjz$, and as previously, we have: 
\begin{eqnarray*}
 \frac{|F(z)|^2e^{-2\vp(z)}}{|z_n|^2 \log^2 |z_n|} &\geqslant & C_3^{-1} \frac{|F(z)|^2}{\log^2 |z_n|}\cdotp \frac 1 {|z_n|^2(e^{2\vp(z',0)}+ |z_n|^{2\alpha})}\\
 & \geqslant & \frac{C_3^{-1}|f(z')|^2}{|z_n|^2 \log^2 |z_n|(e^{2\vp(z',0)}+ |z_n|^{2\alpha})}-\frac{C_3^{-1}C_1}{\log^2 |z_n|(e^{2\vp(z',0)}+ |z_n|^{2\alpha})}
\end{eqnarray*}
We also suppose that $U=U'\times D(0,r_n)$, and we partially integrate with respect to the last variable, in the family of disks $|z_n|<\rho(z')$ with $\rho(z')=\epsilon e^{\alpha^{-1}\vp(z',0)}$ with $\epsilon>0$ small enough so that $\rho(z')<r_n$ for all $z'\in U'$. \\
The right term in the right hand side is easily estimated when integrated, because $\log^2|z_n|\geqslant \log^2 r >0$, $z_n$ being of module $\leqslant r <1$, we have:
\[\int_{|z_n|<\rho(z')}\frac{C_1}{\log^2 |z_n|(e^{2\vp(z',0)}+ |z_n|^{2\alpha})}dV(z_n) \leqslant C_4 \epsilon^2 e^{(\frac 2 {\alpha} -2)\vp(z',0)} \]
which is bounded because $\alpha \leqslant 1$ and $\vp$ is upper bounded.
For the remaining term: 
\begin{eqnarray*}
\int_{|z_n|<\rho(z')} \frac {dV(z_n)} {|z_n|^2 \log^2 |z_n|(e^{2\vp(z',0)}+ |z_n|^{2\alpha})}  & \geqslant & C_5\int_{|z_n|<\rho(z')} \frac{e^{-2\vp(z',0)}dV(z_n)}{|z_n|^2 \log^2 |z_n|}\\
& \geqslant & C_6 e^{-2\vp(z',0)} \int_0^{\rho(z')} \frac{dt}{t \log^2 t}\\
& = & -C_6 \frac{e^{-2\vp(z',0)}}{\log \rho(z')}
\end{eqnarray*}
Then we write $\log \rho(z') = \log \epsilon + \alpha^{-1} \vp(z',0)$, and using the lemma \ref{lem:cru}, the proof is finished. \\
\end{proof}
\begin{rema}
So if we knew that the lemma \ref{lem:cru} still holds under the general assumption that $e^{\vp}$ is Hölder continuous, we would have a general twisted adjunction exact sequence for $\adjz$.
\end{rema}

\bibliographystyle{amsalpha}
\bibliography{Biblio.bib}
%\vspace{3mm}
%\hbox{\raisebox{0.4em}{\vrule depth 0pt height 0.4pt width 4cm}}

\end{document}